\documentclass[12pt]{article}
\usepackage{graphicx}
\usepackage{epstopdf} 
\usepackage{amsmath,amsfonts,amsthm,amssymb}
\usepackage{amsfonts}
\usepackage{amssymb}
\usepackage{graphics}
\usepackage{amsmath,epsfig,psfrag}
\usepackage[english]{babel}
\usepackage[all]{xy}
\usepackage{amscd}
\usepackage{amsthm}
\usepackage{latexsym}
\usepackage{graphicx}
\usepackage{mathrsfs}

\usepackage{amsmath}
\usepackage{amssymb}
\usepackage{latexsym}
\newtheorem{Definition}{Definition}
\newtheorem{Theorem}{Theorem}
\newtheorem{Proposition}{Proposition}

\newtheorem{Corollary}{Corollary}
\newtheorem{Remark}{Remark}
\newtheorem{Question}{Question}
\newtheorem{Example}{Example}

\newtheorem{Claim}{Claim}
\newtheorem{Conjecture}{Conjecture}

\def\TT{{\mathcal T}}

\def\Z{{\mathbb Z}}

\def\S{{\mathbb S}}
\def\F{{\mathbb F}}

\begin{document}
\title{$1$-domination of knots\footnote{Results in Sections 2 and 4 of this note were presented in the
conference "Geometry and Topology of 3-manifolds", ICTP, Trieste,
June 20-24, 2005.}}

\date{\today}
\author{M. Boileau, S. Boyer,  D. Rolfsen,  S.C. Wang}

\maketitle

\begin{abstract}
We say that a knot $k_1$ in the $3$-sphere {\it $1$-dominates} another $k_2$ if there is a proper degree 1 map $E(k_1) \to E(k_2)$ between their 
exteriors, and write $k_1 \ge k_2$. When $k_1 \ge k_2$ but $k_1 \ne k_2$ we write $k_1 > k_2$. One expects in the latter eventuality that $k_1$ is 
more {\it complicated}. In this paper we produce various sorts of evidence to support this philosophy. 

\end{abstract}
\tableofcontents
\section {Introduction}\label{introduction}

All knots are assumed tame and contained in the 3-sphere $S^3$ unless otherwise specified.
For basic terminology in knot theory and 3-manifold theory,
see [Rlf], [He] and [Ja].  If $k \subset S^3$ is a knot, then $N(k)$ denotes a closed
regular neighbourhood and $E(k) = \overline{S^3 \setminus N(k)}$ the knot's exterior.
Fixing an orientation of $S^3$ restricts to a preferred orientation of knot exteriors.

We say that a knot $k_1$ in the $3$-sphere {\it $1$-dominates} another knot 
$k_2$, written $k_1\ge k_2$, if there is a degree $1$ map $f:
E(k_1)\to E(k_2)$ which is proper, i.e. $f(\partial E(k_1)) \subset
\partial E(k_2)$. If $k_1\ge k_2$ but $k_1\ne k_2$, we write $k_1
> k_2$, and say that the $1$-domination is {\it non-trivial} or {\it strict}. Let $l_i,
m_i\subset \partial E(k_i)$ be a longitude-meridian system. We
can assume the $f$ defining the $1$-domination has been homotoped so
that $f|: \partial E(k_1)\to
\partial E(k_2)$ is a homeomorphism and $f(l_1)=l_2$.
The proof of Property P for
knots in $\S^3$ by Kronheimer and Mrowka \cite{KM} implies  
that $f(m_1)=m_2$ if both knots are nontrivial.  In the following we
will assume our maps realizing dominations meet these conditions.

Standard arguments show that the relation $\ge$ provides a partial order on knots in $S^3$.  
Indeed, the transitivity and reflexivity of $\geq$ is clear. For antisymmetry, suppose that $k_1\ge k_2$ and $k_2\ge k_1$. 
Then there are degree $1$ maps $f: E(k_1)\to E(k_2)$ and $g: E(k_2)\to E(k_1)$. 
Since degree $1$ maps are surjective on the level of fundamental groups (cf. Proposition \ref{surjective} below) and knot groups are Hopfian, 
$f$ induces an isomorphism $f_*: \pi_1(E(k_1)) \to \pi_1(E(k_2))$. Then since knot exteriors are Haken and $f$ is a proper map, $E(k_1)$ and $E(k_2)$ are
homeomorphic (cf.\cite{He}  15.13). Finally, since knots are determined by their complements
\cite{GL}, $k_1 = k_2$.

\subsection{Algebraic consequences of $1$-domination}

Suppose that $k_1\ge k_2$ and $\sigma$ is a (set of) knot invariant(s). It is generally believed that  $\sigma(k_1)\; ``\ge" \; \sigma(k_2)$ in some sense, and though this has been verified in various cases,
the general case is unknown. See section \ref{open problems} and \cite{Wan} for discussions.

\begin{Proposition}\label{dominateunknot}  Every knot $1$-dominates the unknot.
\end{Proposition}

\noindent We need to show $k \ge O$ for each knot $k$, where $O$ is the unknot.
Note that any compact manifold $M^n$ with spherical collared boundary,
$\partial M \cong S^{n-1}$, can be degree $1$ mapped onto the ball $B^n$, by pinching
the complement of a collar of $\partial M$ to a point.  If $E(k)$ is a knot exterior,
this trick can be used to map a Seifert surface in $E(k)$ to a spanning disk in $E(O)$,
and another pinch to map the remainder of $E(k)$ to the remainder of $E(O)$. \qed

A similar argument shows the following:

\begin{Proposition}\label{sumdominates}
A connected sum $k_1 \sharp k_2$ of knots $1$-dominates each summand.  
Moreover, if $k_1 \ge k_1'$ and  $k_2 \ge k_1' $ then $k_1 \sharp k_2 \ge k_1' \sharp k_2'.$
\end{Proposition}

Next we consider some invariants which are known to behave well under $1$-domination.  
Proofs of the stated results will be sketched below.

\begin{Proposition}\label{surjective}
If $f: E(k_1) \to E(k_2)$ is a $1$-domination, then $f_* :\pi_1E(k_1) \to \pi_1E(k_2)$ is surjective.
\end{Proposition}


\begin{Proposition}\label{genus}
If $g(k)$ denotes the genus of $k$, then $k_1 \ge k_2 \implies g(k_1)\ge g(k_2)$. 
\end{Proposition}

\begin{Proposition}\label{volume}
If $V(k)$ denotes the Gromov volume of $E(k)$, then $k_1 \ge k_2 \implies V(k_1)\ge V(k_2)$. 
\end{Proposition}

\begin{Proposition}\label{Apoly}
If $A_{k}$ denotes the $A$-polynomial of $k$, then $k_1 \ge k_2 \implies A_{k_2}\vert A_{k_1}$.
\end{Proposition}

\noindent Let  $\Lambda_{k}$ denote the Alexander module associated with the knot $k$.  That is, consider $\widetilde{E(k)} \to E(k)$ the infinite cyclic cover associated with the (kernel of the) Hurewicz map 
$\pi(X) = \pi_1(E(k)) \to H_1(E(k)) \cong \Z$.  Then  $\Lambda_{k}$ is 
$H_1(\widetilde{E(k)}; \Z)$, considered as a $\Z[t^{\pm 1}]$- module, 
where $t$ corresponds to a generator of the deck transformation group $\Z$.

\begin{Proposition}\label{module}
$k_1 \ge k_2 \implies \Lambda_{k_1}=\Lambda_{k_2}\oplus \Lambda$, in
particular $\Delta_{k_2}\vert\Delta_{k_1}$. 
\end{Proposition}

\noindent More generally let $\Delta(k, G) = \{\Delta_{\phi,k}|\,\phi : \pi_1(E(k)) \to G\}$ denote the set
of all {\em twisted} Alexander polynomials  for a given linear group $G$.

\begin{Proposition}\label{twisted}
$k_1 \ge k_2 \implies \Delta(k_2, G)\subseteq \Delta (k_1, G)$. 
\end{Proposition}

The proof of surjectivity of $f_*$ follows from well-known elementary facts and is left to the reader.
Proposition \ref{genus} is a corollary of Gabai's result that embedded Thurston Norms
and singular Thurston Norms coincide \cite{Ga}.
Proposition \ref{volume} is a basic property of Gromov volume, see \cite{Gr} or
 \cite{Th}. A sketch of proof of Proposition \ref{Apoly} can be found in
\cite{SWh}, and also was discussed in a lecture of Boyer \cite{Boy}.
The existence of splittings provided by a degree $1$ map as in Proposition \ref{module} is a classical fact. See \cite{Br} Theorem 1.2.5 for the $\mathbb
Z$-coefficient case and \cite{Wal} p.25 for the local coefficient case. We will present a rather concrete proof based on \cite{Mi} in
Section \ref{Alexander}. For Proposition \ref{twisted} see \cite{KMW}, and also \cite{Z}.

\subsection{Some open problems} \label{open problems}

The behaviour of bridge numbers $b(k)$ under dominations is largely unknown with only partial results currently available \cite{BNW}.
For crossing number $c(k)$, a positive answer to the question of whether $k_1 > k_2$ implies that $c(k_1) > c(k_2)$ would provide an alternative proof of the fact that any knot $1$-dominates at most finitely many knots
\cite{BRuW}. Relatedly, Kitano asked whether $c(k_1)\ge c(k_2)$ if there is an epimorphism $\pi_1E(k_1)\to \pi_1E(k_2)$, which would provide 
an alternate proof of Simon's conjecture \cite{AL}). It would also support the additivity of crossing number under connected sum.

Some flexibility in the interpretation of ``reduces complexity" notion is necessary. For instance for Jones polynomials it is not true that $k_1 \geq k_2$ implies that $V_{k_1}|V{k_2}$ (see the remark after Example 4 in Section \ref{Alexander}), but it is possible that $k_1 \geq k_2$ implies that the degree of  $V(k_1)$ is $\ge$ that of  $V(k_2)$.

More problems will be raised below.

\subsection{Outline of the paper}\label{outline}

In Section \ref{Rigidity}, we will prove some rigidity results about
$1$-domination between knots; that is, under certain conditions, $k
\ge k'$ implies that $k=k'$. Some previously known conditions
include: 
\begin{enumerate}

\item both $k$ and $k'$ are hyperbolic knots and have the
same Gromov volume (Gromov-Thurston's rigidity theorem \cite{Th});

\item  $k$ and $k'$ have the same genus and $k$ is fibred \cite[Corollary 2.3]{BW1}.

\end{enumerate} 
Theorem \ref{rigidity} states that $k\ge k'$ implies
$k=k'$ if $k$ is a knot with no companion of winding number zero,
and if $k$ and $k'$ have the same genus and the same Gromov
volume.  We also construct a strict $1$-domination $k> k'$ such that both
$k$ and $k'$ have same genus, and same Gromov volumes (and same
Alexander polynomials) to show that {Theorem} \ref{rigidity} is a
best rigidity result in terms of genus and Gromov volume. Other results in a similar spirit can be found  in \cite{BNW} and  \cite{De}.

Section \ref{double-cover} is concerned with relations between domination and double branched coverings $M_2(k)$ of $S^3$ over knots $k$.
We show that if $k \ge k'$,  then
\begin{enumerate}

\item[(1)] $M_2(k_1) \ge M_2(k_2)$ (i.e. there is a degree 1 map $M_2(k_1) \to M_2(k_2)$);  

\item[(2)] If $M_2(k_1) = M_2(k_2)$ then $k_1 = k_2$.

\end{enumerate} 
Assertion $(2)$ can be thought of as an extension of the fact that there is no $1$-domination between distinct 
mutant knots with hyperbolic $2$-fold branched coverings \cite{Ru}. We use it to show that knots $1$-dominated by 2-bridge knots, respectively Montesinos knots, are 2-bridge, respectively
Montesinos. (See \cite{ORS}, \cite{Li}, \cite{BB}, \cite{BBRW} for other results on $1$-domination between 2-bridge knots and Montesinos knots). We also show in section \ref{sec:AP} that any knot 
$1$-dominated by a toroidally alternating knot is a connected sum of simple knots. Assertion (1) suggests some interesting questions about the relations between $1$-domination among knots, 
the theory of left orderable groups, and Heegaard-Floer L-spaces. See  \cite{BRoW}, \cite{BGW},  \cite{OS}.

In Section \ref{$1$-domination sequences} we study upper bounds on the length $n$ of
$1$-domination sequences of knots $k_0> k_1>k_2>....>k_n$ with given
$k_0$, which is closely related to rigidity results. It is known
that any sequence of $1$-dominations $M_0>M_1>...>M_i>....$ of compact orientable
3-manifolds has a finite length
\cite{Ro1}, and there is an apriori bound on this length given $M_0$
\cite{So2}. {Theorem} \ref{$1$-domination length in genus} states that if
a knot $k_0$ is free (see Section \ref{$1$-domination sequences} for definitions), then the length of any $1$-domination
sequence of knots $k_0> k_1>k_2>....>k_n$ is bounded by the maximal genus $\hat
g(k_0)$ of an incompressible Seifert surface for $k_0$ when $\hat g(k_0)$ is bounded   . We point out that alternating knots, fibred knots and small knots are
free with bounded $\hat g(k_0)$. If $k_0$ is either fibred or 
2-bridge, then $\hat g(k_0)$ is equal to the genus $g(k_0)$ of $k_0$.  
One-dominations between  small knots, fibred knots, and two bridge knots have also been addressed in \cite{BW2}, \cite{ORS},
 \cite{BB}, \cite{BBRW}.

In Section \ref{Alexander} we present a proof of $\Lambda_{k_1}=\Lambda_{k_2}\oplus \Lambda$ when $k_1 \ge k_2$ 
along with some applications. We also point out that Gordon's approach to 
ribbon concordance \cite{Go},  based on Stallings' results about homology and central
series of groups \cite{Sta}, provides some other rigidity results for
$1$-domination of knots in terms of Alexander polynomials. Consequently, 
the length of a $1$-domination sequence $k_0> k_1
> k_2 > ...> k_n$ of alternating knots is bounded above by the degree of $\Delta_{k_0}$ when its leading coefficient is a
prime power.

It is known that any knot $1$-dominates at most finitely many knots \cite{BRuW}. (See also the stronger results of \cite{AL}, \cite{Liu}.) 
It is very hard to bound the number of knots $1$-dominated by a given knot in general. However the techniques of this paper provide many 
knots which are minimal in the sense that they only $1$-dominate the trivial knot and themselves.

\section{Rigidity via genus and Gromov volume}\label{Rigidity}

\subsection{Satellite knots and an example}\label{satellite}

We recall the definition of satellite knots and fix some notation and terminology needed below.

Suppose that $k_p$ is a knot contained in a solid torus $V$, where
$V \subset S^3$ is unknotted and has longitude and meridian $l,m$.
It is assumed that $k_p$ does not lie in a $3$-ball in $V$. Let
$k_c$ be another knot in $S^3$, with regular neighbourhood
$N(k_c)$, and let $h: V \to N(k_c)$ be a homeomorphism, taking $l$
and $m$ respectively to the longitude and meridian of $k_c$.  
Then the knot $k_s := h(k_p)$ is called the {\it
satellite} of $k_c$ with {\it pattern} knot $k_p$, the latter
considered in $S^3$. One also calls $k_c$ a {\it companion} of
$k_s$.

\begin{Proposition}\label{satellitesdominate}
Satellite knots $1$-dominate their pattern knots.
\end{Proposition}

\begin{proof}
Suppose that $k_s$ is a satellite of $k_c$ with
pattern $k_p$ as described above.  Arguing as in Proposition \ref{dominateunknot} there is a degree $1$ map of the exterior of $k_c$ to the
exterior of $V$.
Combining this with $h^{-1}$ on the closure of $N(k_c) \setminus
N(k_s)$ gives the 1-domination $k_s \geq k_p$. 
\end{proof}

\begin{Example}\label{same volume}
{\rm We construct a non-trivial $1$-domination $k > k_1$ of knots with
the same genus, the same Alexander polynomial, and the same Gromov
volume. Moreover all those invariants are non-vanishing.

Let $k=h(k_1)$ be the satellite of the trefoil $k_2$ indicated by
Figure 2. Here $h: V \to N(K_2)$ is a homeomorphism preserving the
longitudes pictured; $k$ itself is not drawn.  Then we have a
$1$-domination $k\ge k_1$. Let $\TT$
and $\TT_1$ be the JSJ-tori of $E(k)$ and $E(k_1)$ respectively,
then $E(k)\setminus \TT$ consists of three components: two Seifert
pieces and one hyperbolic piece $H$, which is homeomorphic to the
Whitehead link complement; and $E(k_1)\setminus \TT_1$ consists of
two components: one Seifert piece and one hyperbolic piece $H$.
Thus $k > k_1$.  On the other hand, it is clear that both $k$ and $k_1$
are of genus 1, and have the same Gromov volume, which equals the
hyperbolic volume of $H$. They also have the same Alexander
polynomials, since $h$ is longitude preserving (see \cite{Rlf}, Chap.7)
and $k_1$ is an untwisted double.}
\end{Example}

\begin{center}
\includegraphics[totalheight=5cm]{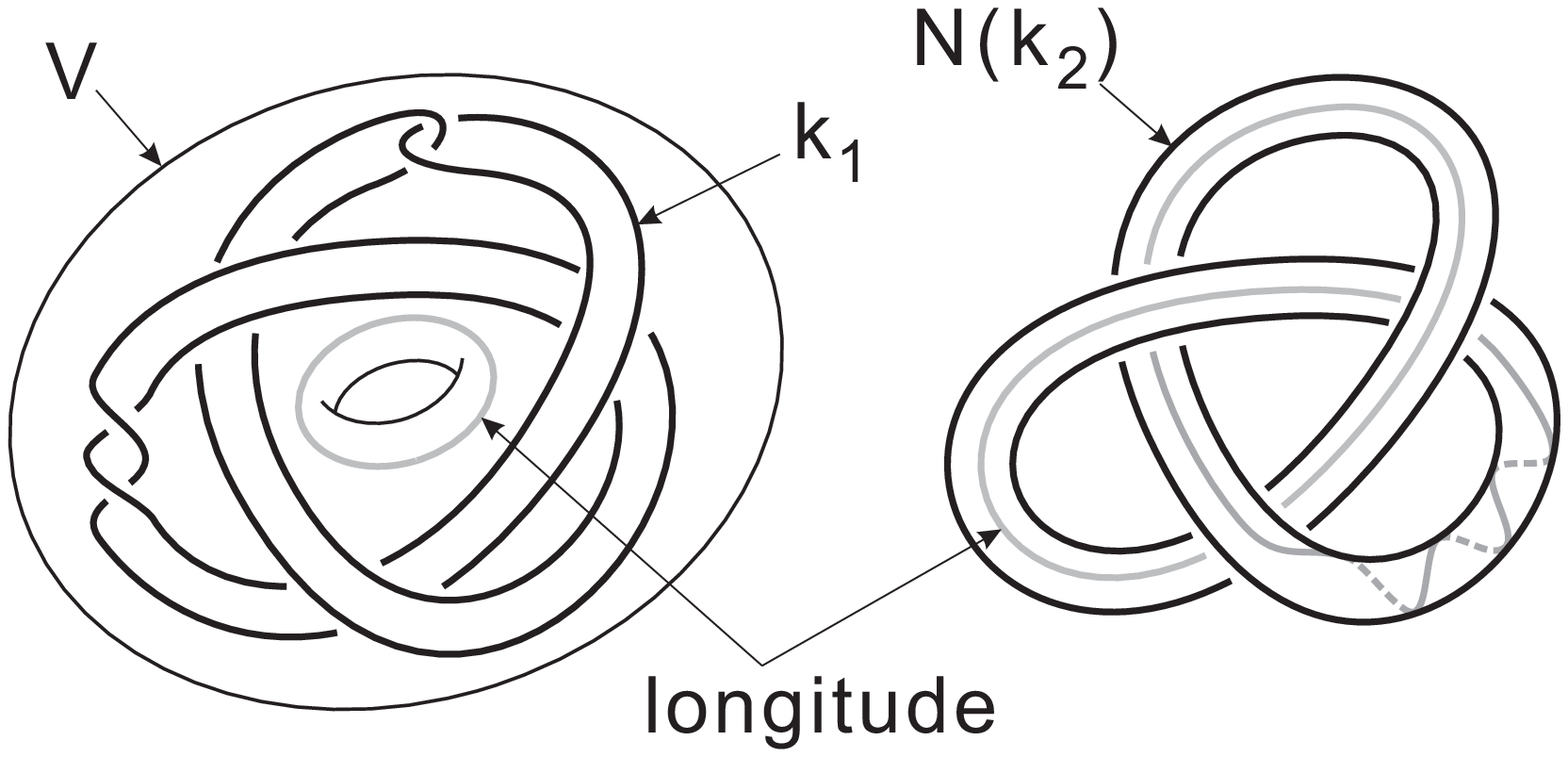}
\begin{center}
Figure 2
\end{center}
\end{center}

\bigskip

\begin{Remark}
{\rm By iterating the construction in Example \ref{same volume}, one
can provide an arbitrarily long $1$-domination sequence of knots
with the same genus, the same Alexander polynomial and the same
Gromov volume.}
\end{Remark}

Suppose $k$ is a knot and $T$ is an essential torus
in $E(k)$. By a theorem of Alexander, $T$ bounds a solid torus $V \cong S^1 \times D^2$, and as 
$T$ is incompressible in $E(k)$, we must have $k \subset V$. Thus $k$ represents some
multiple of the generator of $\pi_1(V) \cong \Z$.  We call the absolute value of this
multiple the {\it winding number} of $T$ relative to $k$. In this setting,
the core curve of $V$ is a companion of $k$.

The essential feature permitting the construction of satellites with the same genus, Alexander polynomial and Gromov volume 
is that the winding number of $k$ in $N(k_2)$ is zero. This turns out to be necessary to the construction, as the following theorem demonstrates.

\begin{Theorem}[Rigidity] \label{rigidity}  
Suppose that $k$ is a non-trivial knot
such that every essential torus
in $E(k)$ has non-zero winding number.
If $k$ and $k'$ have the same Gromov volume and the same genus, and $k\ge k'$, then $k = k'$.
\end{Theorem}

\subsection{Proof of Theorem \ref{rigidity}}

We prove Theorem \ref{rigidity} by establishing a sequence of claims.

\begin{Claim}\label{seifert surface} Let $f: E(k)\to E(k')$ be a  degree 1 map and let $(S, \partial S)\subset (E(k),
\partial E(k))$ be a Seifert surface of minimal genus $g(k)$. Then the
restriction $f|_* : \pi_1(S)\to \pi_1(E(k'))$ is injective.
\end{Claim}

\begin{proof} 
Otherwise there is an essential closed curve $c \subset S$
which is in the kernel of $f|_*: \pi_1(S)\to \pi_1(E(k'))$. Fix a
finite covering $p: \tilde S \to S$ of degree $d$, say, so that
$c$ can be lifted to a simple closed curve $\tilde c$ in $\tilde
S$ (\cite{Sc}). Since $f(S)$ carries a generator $a$ of
$H_2(E(k'), \partial E(k'); \mathbb Z)$ and $g(k)= g(k') > 0$, the
Thurston Norm of $a$ is $|\chi(S)|$.  It follows that $(f \circ
p)(\tilde S)$ carries $da$ and realizes its Thurston norm, which is
$|\chi(\tilde S)|=d|\chi(S)|$. However since the simple essential
closed curve $\tilde c$ lies in the kernel of $(f\circ p)_*$, we can
perform surgery on $\tilde S$ along $\tilde c$ to produce a new
surface $\tilde S^*$ and a map $g: \tilde S^*\to E(k')$ which also
represents $da$. But then the singular Thurston norm of $da$ is bounded above by
$|\chi(\tilde S^*)|$, which is strictly less than $|\chi(\tilde S)|$, contrary to Gabai's result
that the Thurston norm and singular Thurston norm coincide. 
\end{proof}

\begin{Claim}\label{injective} If $T\subset E(k)$ is any essential torus, 
then the restriction $f|_* : \pi_1(T)\to \pi_1(E(k'))$ is injective.
\end{Claim}

\begin{proof} 
Let $k_T$ be the companion of $k$ such that $\partial
E(k_T)=T$, and let $(m_T, \ell_T)$ be the meridian-longitude pair of $k_T$ 
on $\partial E(k_T)$. If $w_{T}$ denotes the winding number of
$T$, one has $m_T= w_{T}m$ in $H_1(E(k); \mathbb Z)$and so for any integers $p$ and $q$,
$p\ell_T + qm_T= qw_{T} m$ in $H_1(E(k); \mathbb Z)$. Since $f_*:
H_1(E(k); \mathbb Z)\to H_1(E(k'); \mathbb Z)$ is an isomorphism
given by $f_*(m)=m'$ and $\pi_1(E(k'))$ is torsion free, it
follows that if the kernel of $f|_* : \pi_1(T)\to \pi_1(E(k'))$
is non-trivial, then it is generated by the longitude $\ell_T$ on
$T$.  As argued by Schubert, any minimal Seifert surface $S$ for $k$ may
be assumed to intersect $T$ in $w_T$ longitudes. Since by hypothesis $w_T\ne
0$, we may assume $\ell_T\subset S$ and represents a nontrivial
element of $\pi_1(S).$  But $f|_* :\pi_1(S)\to \pi_1(E(k'))$ is
injective by Claim \ref{seifert surface}, so $f_*(\ell_T)\ne 1$.
Claim \ref{injective} is proved. 
\end{proof}

\begin{Claim}\label{seifert piece} If $N\subset E(k)$ is a Seifert piece of the
JSJ-decomposition of $E(k)$, then the restriction $f|_* : \pi_1(N)\to \pi_1(E(k'))$
is injective.
\end{Claim}

\begin{proof} 
It follows from Seifert's classification of Seifert fibre structures on $S^3$ that
$N$ is either a torus knot exterior, a cable space, or a composing space
(the product of a planar surface and a circle) with at least three boundary components
(see Lemma VI.3.4 of \cite{JS}). In particular, its  base orbifold is orientable and
therefore $N$ admits no separating, horizontal surfaces.

Let $T \subseteq \partial N$ be either $\partial E(k)$ or the torus which separates $N$ from $\partial E(k)$ 
and fix a minimal genus Seifert surface $S$ for $k$. Assume that $S$ has been isotoped
to intersect $\partial N$ minimally and recall from the proof of the previous claim that
$S \cap T$ consists of $w_T > 0$ copies of the longitude $\ell_T$. Fix a component
$S_0$ of $S \cap N$ such that $S_0 \cap T \ne \emptyset$. Clearly $S_0$ is an essential
surface in $N$ and so can be assumed to be either vertical or horizontal with respect to a
fixed Seifert structure on $N$. Now $\ell_T$ cannot be isotopic in $T$ to a Seifert fibre of
$N$ (this can verified for each of the three types of possibilities for $N$), so $S_0$
is horizontal and therefore non-separating in $N$. It follows that $N$ fibres over the circle
with fibre $S_0$. By Claim \ref{seifert surface}, $f_*|\pi_1(S_0)$ is injective and so
$f_*(\pi_1(S_0))$ is a non-abelian free group.
(It follows from the previous paragraph that $\chi(S_0) < 0$.)

Recall that the class $\phi$ of a regular fibre of $N$ is central in $\pi_1(N)$ and 
let $H$ be the group generated by $\phi$ and $\pi_1(S_0)$.
Then $H$ has finite index in $\pi_1(N)$ and since the latter is torsion free,
it suffices to show that $f_*|H$ is injective. An element of $H$ can be written
$\gamma \phi^n$ for some $\gamma \in \pi_1(S_0)$ and $n \in \mathbb Z$. Thus if
$f_*(\gamma \phi^n) = 1$, then $f_*(\phi)^n = 1$ since it is a central element of the
non-abelian free group $f_*(\pi_1(S_0))$. But $f_*(\phi) \ne 1$ by Claim \ref{injective}, 
and since  $\pi_1(E(k'))$ is torsion free we see that $n = 0$.
Then $f_*(\gamma) = 1$ so that $\gamma \phi^n = \gamma = 1$. Thus the Claim holds.
\end{proof}

\vskip 0.5 true cm

Let $E(k) = H_k \cup S_k$ and $E(k') = H_{k'} \cup S_{k'}$ where $H_k, H_{k'}$ and $S_k, S_{k'}$ are the unions of the hyperbolic and Seifert pieces of $E(k)$ and $E(k')$.

\begin{Claim}\label{hyperbolic piece} The map $f$ can be homotoped so
that:

\noindent $(1)$ $f|: (H_k, \partial H_k) \to (H_{k'}, \partial H_{k'})$ is a homeomorphism.

\noindent $(2)$ $f(S_k) = S_{k'}$.
\end{Claim}

\begin{proof} 
Define $\Sigma_k$ to be the union of $S_k$ and regular
neighbourhoods of the characteristic tori connecting two
hyperbolic pieces in the JSJ decomposition of $E(k)$. Define
$\Sigma_{k'}$ similarly. By Claim 3 and the enclosing property of
characteristic submanifold theory (\cite{JS}), we may homotope
$f|\Sigma_k$ into $\Sigma_{k'}$. If $\partial E(k) \subset
\Sigma_k$ we may suppose that the homotopy leaves $f|\partial
E(k)$ invariant. Extend this homotopy to a homotopy of $f$
supported in a regular neighbourhood of $\Sigma_k$. Since the
Gromov norms of $E(k)$ and $E(k')$ are the same, by Soma's result \cite{So1} one can further
modify $f$ by a homotopy fixed on $S_k$ so that $f|H_k$ is a
homeomorphism $(H_k, \partial H_k) \to (H_{k'}, \partial H_{k'})$. Then $f^{-1}(S_{k'}) \subset S_k$ and since $f$ is
surjective we have $f(S_k) =S_{k'}$.
\end{proof}

\begin{Claim}\label{seifert preimage} Distinct neighbouring Seifert pieces of $E(k)$ are sent to distinct neighbouring Seifert pieces of $E(k')$ by $f$. Further, if $N' \subset S_{k'}$ is a Seifert piece of $E(k')$, then $f^{-1}(N')$ is a Seifert piece of $E(k)$.
\end{Claim}

\begin{proof} 
Suppose that there are distinct but non-disjoint Seifert pieces $N_1, N_2$ of $E(k)$ which are sent into $N'$ by $f$. Since $S^3$ is simply-connected, $N_1 \cap N_2$ is a torus $T$ and if $\phi_1, \phi_2 \in \pi_1(T) \cong \mathbb Z^2$ represent the fibre classes of $N_1, N_2$ respectively, they generate a $\mathbb Z^2$ subgroup of $\pi_1(T)$. Claim \ref{injective} shows the latter statement also holds for  $f_*(\phi_1), f_*(\phi_2)$. On the other hand, Claim \ref{seifert piece} implies that $f_*(\phi_j)$ has a non-abelian centralizer in $\pi_1(E(k'))$ and so Addendum VI.1.8 of Theorem VI.1.6 in  \cite{JS} implies that $f_*(\phi_j)$ is a power of the fibre class of $N'$ ($j = 1, 2$). But then $f_*(\phi_1)$ and $f_*(\phi_2)$ lie in a $\mathbb Z$ subgroup of $\pi_1(E(k'))$, which we have seen is impossible. Thus distinct neighbouring Seifert pieces of $E(k)$ are sent to distinct neighbouring Seifert pieces of $E(k')$ by $f$.

The dual graph $\Gamma(k)$ to the JSJ-decomposition of $E(k)$ is a rooted tree where the root vertex $v_0$ corresponds to the vertex manifold containing  $\partial E(k)$. For each vertex $v$ of $\Gamma(k)$, we use $X_v$ to denote the corresponding vertex manifold. Define $\Gamma(k'), v_0'$, and $X_{v'}$ similarly. Since $S^3$ is simply-connected, both $\Gamma(k)$ and $\Gamma(k')$ are trees. Hence, Claim \ref{hyperbolic piece} and the conclusion of the previous paragraph imply that $f$ induces an isomorphism between these trees, which proves the claim.
\end{proof}

\medskip
Claims \ref{hyperbolic piece} and \ref{seifert preimage} imply that for each piece $X'$ of $E(k')$, there is a unique piece $X$ of $E(k)$ such that $f: (X, \partial X) \to (X', \partial X')$. If $v$ is the vertex of $\Gamma(k)$ corresponding to $X$, we let $f(v)$ denote the vertex of $\Gamma(k')$ corresponding to $X'$.

Theorem \ref{rigidity} is a consequence of our final claim:

\begin{Claim}\label{winding number non-zero}
The map $f$ can be homotoped to a homeomorphism.
\end{Claim}

\begin{proof} 
Given the conclusions of Claims \ref{injective}, \ref{seifert piece} and \ref{hyperbolic piece}, classic work of Waldhausen shows that the restriction $f|: (X_{v_0}, \partial X_{v_0}) \to (X'_{v_0'}, \partial X_{v_0'}')$ is homotopic (rel $\partial E(k)$) to a covering map (see Theorem 13.6 of \cite{He}). Since $f| \partial E(k)$ has degree $1$, $f|: (X_{v_0}, \partial X_{v_0}) \to (X'_{v_0'}, \partial X_{v_0'}')$ can be homotoped to a homeomorphism. (This is automatic of course if $X_{v_0}$ is hyperbolic.) In particular $|\partial X_{v_0}| = |\partial X'_{v_0'}|$.

Now suppose that the vertices of $\Gamma(k)$ (respectively $\Gamma
(k')$) adjacent to $v_0$ (respectively $v'_0$) are $v_1,...v_p$ (respectively
$v_1',...,v_p'$). Let $T_i$ be the torus $X_{v_0} \cap X_{v_i}$, $1 \leq i \leq p$ and $T_i'$ its image by $f$. By Claim \ref{seifert preimage}, $f$ cannot send $X_{v_i}$ to $X'_{v'_0}$, $i \ne 0$, so we may assume that $f(X_i)\subset X'_{v'_i}$ for $i = 1, \dots , p$. Since $f|: T_i\to T'_i$ is a homeomorphism, the argument of the previous paragraph shows that for each $i$, $f: X_{v_i} \to X'_{v'_i}$ is homotopic to a homeomorphism (rel $T_i$). Proceeding by induction we see that for each vertex $v$ of $\Gamma(k)$, $f|: (X_v, \partial X_v) \to (X'_{f(v)}, \partial X'_{f(v)})$ is a homeomorphism. Since $f$ induces an isomorphism $\Gamma(k) \to \Gamma(k')$, the proof of the claim is complete. 
\end{proof}

\section{Double branched covers}\label{double-cover}

Let $M_q(k)$ denote the $q$-fold cyclic branched covering of $S^3$ over the knot $k$.

\begin{Theorem}\label{double cover}
Suppose $k\ge k'$.  Then

\noindent $(1)$ $M_2(k)\ge M_2(k')$, that is, a degree 1 map $M_2(k)\to M_2(k')$ exists; and 

\noindent $(2)$  $M_2(k)=M_2(k') \implies k=k'$.
\end{Theorem}

\begin{proof}
 Suppose that there is a degree $1$ map $f: E(k_1)\to E(k_2)$.
 We may assume that  $f|: \partial E(k_1)\to \partial E(k_2)$ is a  homeomorphism which sends $m_1$ to $m_2$.
 
(1) Pick a Seifert surface $F'$ of $k'$. We may assume that $f$ has been homotoped relatively to the boundary
to be transverse to $F'$ and so that $F=f^{-1}(F')$ is connected. Then $F$ is a Seifert surface
of $k$. Then $f$ restricts to a proper degree $1$ map between $E(k)$ cut open along $F$, which we denote by $E(k)\setminus F$, and $E(k')\setminus F'$. 
This restriction can be assumed to be a homeomorphism between the two copies of $F$ in $\partial(E(k)\setminus F)$ to the two copies of $F'$ in $\partial(E(k')\setminus F')$.
By gluing two copies of $E(k)\setminus F$, respectively $E(k')\setminus F'$, along these copies of $F$,
respectively of $F'$, we obtain a degree $1$ map from the $2$-fold covering of $E(k)$ to the $2$-fold cyclic
covering of $E(k')$, which extends to a degree $1$ map $\hat f$ from $M_2(k)$ to $M_2(k')$.
 In other words, the degree $1$ map $f: E(k_1)\to E(k_2)$ lifts to a degee 1 map between the 2-fold coverings of the knot exteriors, which extends to a degree $1$ map $\hat f: M_2(k) \to M_2(k')$.

(2)  A degree $1$ map, $f$ induces an epimorphism $f_*: \pi_1E(k_1)\to \pi_1 E(k_2)$ such that $f_*(m_1) = m_2)$, hence it induces an epmorphism 

 $$\bar f_*: \pi_1E(k_1)/m_1^2\to \pi_1 E(k_2)/m_2^2$$
 and we have the commutative diagram

$$\xymatrix{
  1  \ar[r] & \pi_1( M_2(k)) \ar[d]_{\hat f_*}\ar[r] & \pi_1(E(k)/m_1^2)\ar[d]_{\bar f_*} \ar[r] & \Z_2\ar[d]_{\cong}\ar[r] & 1\\
  1  \ar[r] & \pi_1(M_2(k')) \ar[r] & \pi_1(E(k')/m^2_2) \ar[r] & \Z_2\ar[r] & 1
  }
$$

Since $\hat f_*$ is induced by a degree $1$ map, it is surjective. Therefore $\hat f_*$ is an isomorphism, because 
$M_2(k)=M_2(k')$ and 3-manifold groups are Hopfian.
By the Five Lemma, $\bar f_*$ is also an isomorphism. Assertion (2) follows now from the geometrization of 3-orbifolds with singular locus a link \cite{BP}, a theorem of Boileau-Zimmermann 
about $\pi$-orbifolds groups \cite{BZ} when $\pi_1(M_2(k))$ is infinite and the classification of spherical Montesinos knots when $\pi_1(M_2(k))$ is finite.
\end{proof}

Mutant knots have the same double branched covering, so Proposition \ref{double cover} implies the following.

\begin{Corollary}\label{cor:mutant}
There is no $1$-domination between distinct mutant knots.
\end{Corollary}
 
\noindent Ruberman has shown that if $k, k'$ are mutants, then $E(k)$ is hyperbolic if and only if $E(k')$ is, and in this case, both have the same volume \cite{Ru}. 
Hence in this situation, Corollary \ref{cor:mutant} follows from the Gromov-Thurston rigidity theorem. 

Assertion $(1)$ of Theorem \ref{double cover} provides a connection between $1$-domination among knots, left orderable groups and L-spaces.

\begin{Definition}  $\;$ \\ 
{\rm $(1)$ A group is {\it left-orderable} if there is a total ordering $<$ of its elements which is left-invariant: $x < y$ if and only if $zx < zy$ for all $x, y$ and $z$.

\noindent $(2)$ An {\it L-space} is a closed rational homology $3$-sphere whose Heegaard-Floer homology $\widehat {HF}(M)$  is a free abelian group of rank 
 equal to $|H_1(M,\Z)|$. }
\end{Definition}

\begin{Proposition}\label{left orderable} {\rm (\cite{BRoW})}
Suppose $G$ and $G'$ are nontrivial fundamental groups of irreducible 3-manifolds and there is a surjection $G\to G'$.
 If $G'$ is left orderable, then $G$ is left orderable.
\end{Proposition}

\begin{Corollary} \label{lodomination}
If $\pi_1M_2(k_1)$ is not left orderable but $\pi_1M_2(k_2)$ is, then $k_1$ does not $1$-dominate $k_2$.
\end{Corollary}

\noindent The left orderabilty of $\pi_1M_2(k)$ can be determined for certain family of knots. For instance, Boyer-Gordon-Watson showed that this is never the case for non-trivial alternating knots $k$ \cite{BGW}. 
For each Montesinos knot $k$,  $M_2(k)$ is a Seifert manifold, and work of Boyer-Rolfsen-Wiest \cite{BRoW} combines with that of Jankins-Neumann \cite{JN} and Naimi \cite{Na} to determine exactly when such manifolds have left orderable fundamental groups in terms of the Seifert invariants. As a consequence, alternating knots cannot $1$-dominate certain classes of Montesinos knots.

Another result, due to Ozsvath-Szabo, states that $M_2(k)$ is an L-space for each 
alternating knot $k$  \cite{OS}.  This and other evidence corroborates the following conjecture in \cite{BGW}, which is unsolved at this writing.

\begin{Conjecture} \label{bgwconj} 
 An irreducible $3$-manifold which is a rational homology sphere is an L-space if and only if its fundamental group is not left orderable. 
 \end{Conjecture}

\noindent Ozsv\'ath-Szab\'o have conjectured that an irreducible $\mathbb Z$-homology $3$-sphere is an L-space if and only if it is the $3$-sphere or the Poincar\'e homology sphere (cf. \cite[Problem 11.4 and the remarks which follow it]{Sz}). This combines with Conjecture \ref{bgwconj} to yield the following conjecture: {\it  An irreducble $\mathbb Z$-homology $3$-sphere other than $S^3$ and the Poincar\'e homology sphere has a left-orderable fundamental group}. 

Recall that the {\it determinant} of a knot $k$ is given by $|\Delta_{k}(-1)|$ and coincides with $|H_1(M_2(k),\Z)|$. Thus $M_2(k)$ is a $\mathbb Z$-homology $3$-sphere if and only if the determinant of $k$ is $1$. The discussion above leads to the following question, whose expected answer is no. 

\begin{Question}
Suppose that $k$ is alternating. 
Can $k$ $1$-dominate a nontrivial knot $k'$ with $|\Delta_{k'}(-1)|=1$? 
In particular a nontrivial knot with trivial Alexander polynomial?
\end{Question}

\noindent Here is a related question. 
\begin{Question}
Suppose that $k$ is alternating and $k\ge k'$. Does $|\Delta_{k}(-1)|=|\Delta_{k'}(-1)|$
imply that $k=k'$? 
\end{Question}

To state our next results, we need to recall some definitions: a {\it 2-string tangle} is the 3-ball $B^3$ with two disjoint properly embedded arcs $a_1 \cup a_2$. 
A {\it trivial tangle} 
is a 2-string tangle where the  arcs $a_1$ and $a_2$ bound disjoint disks
together with arcs on the boundary of $B^3$.
A {\it rational tangle} is the image of a trivial tangle by a homeomorphism of the ball fixing the end points of  the arcs $a_1$ and $a_2$: a tangle is rational if and only if the $2$-fold covering of $B^3$ branched along the arcs $a_1 \cup a_2$ is a solid torus $S^1 \times D^2$.
A well-known fact, using this double branched covering, is that rational tangles correspond to rational numbers,
called the {\it slopes} of the rational tangles: a rational tangle $T(r)$ corresponding to the rational number $r$ is obtained by first drawing
two strings of slope $r$ on the boundary $S^2(2,2,2,2)$ of the pillow-case $B$, then pushing into its interior.
A {\it Montesinos tangle} is a tangle sum of  rational tangles $T(r_1),..., T(r_n)$:
by adding two arcs on the boundary of $B$, we get the so-called 
2-bridge knots and Montesinos knots.
The double branched cover of those knots are, respectively, lens spaces and Seifert manifolds. Further, the action of the covering involution $\tau$ preserves the Seifert fibre of  the $2$-fold covering and reverses its orientation. The converse is true by the orbifold theorem \cite{BP}, see also \cite{BS} : If $M_2(k)$ is, respectively, a lens space or a
Seifert fibered manifold and $\tau$ reverses the orientation of the Seifert fibre, then 
$k$ is, respectively, a 2-bridge knot or a Montesinos knot.

\begin{Proposition}\label{Montesinos}  
Suppose $k\ge k'$. 

\noindent $(1)$ If $k$ is a 2-bridge knot, so is $k'$.

\noindent $(2)$ If $k$ is a Montesinos knot, so is $k'$.

\end{Proposition}

\begin{proof}
Let $f: E(k)\to E(k')$ be a $1$-domination and $\tau$, $\tau'$ are the covering involutions of the 2-fold branched coverings
$M_2(k)$ and $M_2(k')$.
Then we have a $\Z_2$ equivalent  degree 1 map $\tilde f : M_2(k) \to M_2(k')$ , i.e. 
$\tau\circ \tilde f= \tilde f \circ \tau'$ and a surjection
$\tilde f_*: \pi_1M_2(k)\to \pi_1 M_2(k')$.  

If $k$ is a 2-bridge knot, $M_2(k)$ is a lens space and therefore $\pi_1(M_2(k))$ is a finite cyclic group. Hence 
$\pi_1(M_2(k'))$ is finite cyclic so by the orbifol theorem \cite{BP}, $M_2(k')$ is a lens space and $k'$ is $2$-bridge. 
Next suppose that $k$ is a Montesinos knot, so that $M_2(k)$ is an irreducible Seifert manifold.  
Again it is known that $k'$ is Montesinos if $\pi_1M_2(k')$ is finite by the orbifold theorem \cite{BP}, so suppose otherwise. Then $\pi_1M_2(k)$ is 
infinite and non-cyclic, so $M_2(k)$ is a $K(\pi_1M_2(k), 1)$ space, finitely covered by a circle bundle $W$ over a closed orientable surface $F$ where the circle 
fibres of $W$ are the inverse image of the Seifert fibres of $M_2(k)$. The class $h$ of a regular fibre of $M_2(k)$ cannot be contained in the 
kernel of $\tilde f_*$ as otherwise the composition $W \to M_2(k) \to M_2(k')$ would factor through $F$, which is impossible for a non-zero degree map.   

Suppose that $M_2(k')$ is reducible and let $S'$ be an essential 2-sphere it contains. After a homotopy of $\tilde f$ we can suppose that the 
preimage $S$ of $S'$ in $M_2(k)$ is an essential surface. Now $S$ cannot be vertical as this would imply that the $h$ would be contained in the 
kernel of $\tilde f_*$. On the other hand it cannot be horizontal in $M_2(k)$ since this would imply that the odd-order abelian group 
$H_1(M_2(k); \Z)$ has a $\Z_2$ quotient. Thus $M_2(k')$ is irreducible. It follows that $\pi_1M_2(k')$ is torsion-free and therefore $\tilde f_*(h)$ has infinite order.  
But then $\pi_1M_2(k')$ has a non-trivial centre containing $\tilde f_*(h)$, so $M_2(k')$ is a Seifert manifold by \cite{CJ}, \cite{Ga2}  with $\tilde f_*(h)$ a non-trivial power of the fibre class. It follows that $k'$ 
is either a Montesinos knot or a torus knot. The former case happens if $\tau'$
reverses orientation of each fibre, and the latter case otherwise. Since $\tau'_*(\tilde f_*(h)) = \tilde f_*(\tau_*(h)) = \tilde f_*(h^{-1}) = \tilde f_*(h)^{-1}$, $k'$ is a Montesinos knot. 
\end{proof}

\begin{Remark}
{\rm Many Seifert manifolds are known to be minimal (see \cite{HWZ} for example) from which we can deduce that many Montesinos knots are minimal.}
\end{Remark}

\section{AP-property and 1-domination}\label{sec:AP}
 
A knot $k$ has the {\em AP-property} if every closed incompressible surface embedded in its complement carries an essential closed curve homotopic to a peripheral element. (Here $AP$ refers to {\it accidental parabolic}.) Small knots (i.e.~knots whose exteriors contain no closed essential surfaces) are AP knots, but so are {\it toroidally alternating knots} \cite{Ad}, a large class which contains, for instance, all hyperbolic knots which are alternating, almost alternating, or Montesinos.

A knot $k$ is {\em simple} if it is either a hyperbolic or a torus knot. This condition is equivalent to the requirement that $E(k)$ contain no essential tori. 

\begin{Proposition}\label{AP} Let $k$ be a knot with the AP-property. Then $k$ can dominate only a connected sum of simple knots or a cable of a simple knot.
\end{Proposition}

\begin{proof}
Suppose that $k$ dominates a satellite knot $k'$ and let $T'$ be a JSJ torus in $E(k')$ which bounds a simple knot exterior (i.e. is innermost). Fix a degree $1$
map $f: E(k)\to E(k')$ which is transverse to $T'$. Since $E(k')$ is irreducible, $f$ can be homotoped so that
each component $F$ of $f^{-1}(T')$ is incompressible in $E(k)$. 
The AP-property implies that some essential closed curve $\gamma \subset F$ is freely homotopic to an essential closed curve 
$\alpha \subset \partial E(k)$. Up to replacing $F$ by another component of $f^{-1}(T')$, we can assume that the homotopy takes place 
in the outermost component of $\overline{E(k) \setminus f^{-1}(T')}$ (i.e. the component which contains $\partial E(k)$). 
Applying $f$ we obtain a homotopy in $W$, the outermost component of $\overline{E(k') \setminus T'}$. Since the restriction $f|: \partial E(k) \to \partial E(k')$ 
is a homeomorphism, $f(\alpha)$ is an essential closed curve on  $\partial E(k')$, and therefore the annulus theorem \cite{JS} provides an essential annulus $A$  
properly embedded in $W$ and cobounded by essential simple closed curves on $T'$ and $\partial E(k')$. We can assume that $A$ intersects the JSJ tori of $E(k')$ transversely and minimally. 
Then it intersects each of the JSJ pieces it passes through in essential, properly embedded annuli. It follows that these pieces are Seifert fibred. Further, since $S^3$ contains no Klein bottles, the annuli are vertical in their respective pieces, so their Seifert structures match up. Thus $W$ is Seifert fibred and hence a piece of $E(k')$. Since $T'$ was an arbitrary innermost JSJ torus of $E(k')$, the result follows. 
\end{proof}

\begin{Corollary}\label{connected sum} A toroidally alternating knot can dominate only a connected sum of simple knots. In particular this holds for alternating knots. 
\end{Corollary}

\begin{proof}  
Let $k$ be a toroidally alternating knot which $1$-dominates a knot $k'$, say $f:E(k) \to E(k')$ is a degree $1$ map. We have assumed that $f$ restricts to a homeomorphism $\partial E(k) \to \partial E(k')$ and hence sends a meridian of $k$ to an essential simple closed curve on $\partial E(k')$ which normally generates $\pi_1E(k')$. Thus the Property P conjecture \cite{KM} shows that $f$ sends the meridian of $k$ to the meridian of $k'$. By Proposition \ref{AP}, $k'$ is either a product of simple knots or a cable of a simple knot. Let $W$ denote the outermost JSJ piece of $E(k')$. 

Each closed incompressible surface embedded contained in $E(k)$ carries an essential simple closed curve homotopic to a meridian of $k$ by \cite[Corollary 3.3]{Ad}, so the proof of Proposition \ref{AP} shows that there is a homotopy in $W$ between a meridian of $k'$ and an essential loop on an innermost JSJ torus of $E(k')$. But this never occurs when $W$ is a cable space, so $k'$ is a product of simple knots. 
\end{proof}

For Montesinos knots Corollary \ref{connected sum} follows from Proposition \ref{Montesinos}, since a Montesinos knot is simple by \cite{Oe}.




\section{Length of $1$-domination sequences via genus}\label{$1$-domination sequences}

\begin{Definition}  

\noindent $\;$ \\ 
{\rm \noindent $(1)$ We recall that a knot is {\it small} if each incompressible
closed surface in its exterior is boundary parallel.

\noindent $(2)$ A Seifert surface $S$ of a knot $k$ is {\it free} if
$\overline{E(k)\setminus T(S)}$ is a handlebody where $T(S)$ is a tubular neighbourhood of $S$.  A knot $k$ is {\it free},
if all its incompressible Seifert surfaces are free. For example a small knot is free.

\noindent $(3)$ Define $\hat g (k)=\text{sup} \{ g(S) | \text{$S$ is an
incompressible Seifert surfaces for $k$}\}.$  Here $g(S)$ denotes the
genus of the surface $S$.}
\end{Definition}

Classic pretzel knots with three branches are small \cite{Oe} as are 2-bridge knots [HT]. Fibered knots are free, which follows directly from the classical result that each 
incompressible Seifert surface of a fibred knot is isotopic to its fibre surface. In particular, $\hat g(k)=g(k)$ for fibred knots. If $k$ has a companion of winding number
zero, then $k$ is not free since there is an incompressible Seifert surface for $k$ which is contained in the complement of the companion torus.

Clearly $g(k) \le \hat g(k )\le \infty$ and it is possible that $\hat g(k)=\infty$ (see [Ly]). Small knots satisfy $\hat g(k) < \infty$ by \cite{Wi}. 
Non-fibred examples of knots for which $\hat g(k)=g(k)$ include non-fibred  2-bridge knots, ([HT]).  

\begin{Question} Which knots $k$ in $S^3$ have bounded $\hat g(k)$ and which are free?
 \end{Question}

\noindent Here is a construction which produces several interesting classes of free knots.

\begin{Proposition} \label{small-free}
A knot $k$ with the property that each closed, essential surface in $E(k)$ contains a loop which links $k$ homologically a non-zero number of times is free.
\end{Proposition}

\begin{proof} Suppose that $k$ is not free and choose an incompressible Seifert surface for $k$ whose complement is not a handlebody. Set $H= \overline{E(k)\setminus T(S)}$ and consider a maximal compression body $P$ for $\partial H$ in $H$. There is a decomposition
$$H = P \cup V = (\partial H \times I) \cup \text{2-handles} \cup V$$
where $V$ is a not necessarily connected, compact $3$-manifold. Since $P$ is maximal, $\partial V$ either is a finite union of 2-spheres or has a component which is incompressible in $V$. In the former case, the incompressibility of $S$ and irreducibility of $E(k)$ implies that $V$ is a finite union of $3$-balls. But then $H$ is a handlebody, contrary to our assumptions. Thus there is a component $F$ of $\partial V$ which is incompressible in $H$ and therefore in $H = V \cup  (\text{1-handles})$ and $E(k) = H \cup (S \times I)$. Since $S$ is contained in $\overline{E(k) \setminus V}$ and is not $\partial$-parallel in $E(k)$, $F$ is essential in $E(k)$. By construction, $F \cap S = \emptyset$ and so every loop on $F$ links $k$ zero times, contrary to our hypotheses. Thus $k$ must be free.
\end{proof}

\begin{Corollary}
Small knots, alternating knots, and Montesinos knots are free.
\end{Corollary}

\begin{proof} Any such knot satisfies the hypotheses of Proposition \ref{small-free} and so is free. This is obvious for small knots. On the other hand, a closed essential surface in the exterior of either an alternating knot or a Montesinos knot $k$ contains a simple closed curve which is homotopic in $E(k)$ to a meridian of $k$ (\cite{Me}, \cite{Oe}), which implies the claim. 
\end{proof}

\begin{Proposition}\label{free-free}
Suppose that $k \ge k'$.

\noindent $(1)$ If $k$ is free, then $k'$ is free.

\noindent $(2)$ If $k$ is free and $f: E(k) \to E(k')$ is a degree 1 map such that $g(S) = g(S')$ 
where $S'$ and $S=f^{-1}(S') $ are incompressible Seifert surfaces for $k', k$,  then $k = k'$.

\noindent $(3)$ $\hat g (k) \ge \hat g (k')$, and if $k$ is free with bounded $\hat g(k)$, then $\hat g (k) = \hat g (k')$ if and only if $k = k'$.

\end{Proposition}

\begin{proof} 
(1) Let $S'$ be an incompressible  Seifert surface of $k'$
with genus $g(k')$, and let $f: E(k)\to E(k')$ be a degree $1$
map, transverse to $S'$, which  realizes the $1$-domination $k\ge
k'$. Since $E(k')$ is irreducible, $f$ can be homotoped so that
each component of $f^{-1}(S')$ is incompressible. Further, since
$f$ has degree $1$, exactly one component $S$ of  $f^{-1}(S')$  is a Seifert surface of $k$ and the remaining components are closed. Since $k$ is a free knot, it follows
that $S = f^{-1}(S')$. 

Let $T(S) \subset
E(k)$ be a tubular neighbourhood of $S$. Then $f$ induces a proper
degree $1$ map
$$f|: H = \overline{E(k) \setminus T(S)} \to \overline{E(k') \setminus T(S')} = H'.$$
\noindent Consider a maximal compression body $P'$ for $\partial H'$ in $H'$. There is a decomposition
$$H' = P' \cup V' = (\partial H' \times I) \cup \text{2-handles} \cup V'$$
where $V'$ is a not necessarily connected, compact $3$-manifold. Since $P'$ is maximal, $\partial V'$ either has a component $F'$ which is incompressible in $V'$   or is a finite union of 2-spheres. In the former case,  $H'$ contains closed incompressible surface $F'$, and
 $f|H$ could be homotoped rel $\partial H$ to a function $g$ such that $g^{-1}(F)$ is a closed and essential in $H$, contrary to the hypothesis that $H$ is a handlebody. Hence the latter case arises  and the incompressibility of $S'$ and irreducibility of $E(k')$ implies that $V'$ is a finite union of $3$-balls.
This shows that $H'$ is a handlebody,  and completes the proof of (1).

(2) By hypothesis, $f|: S \to S'$ is a proper degree $1$ map between homeomorphic surfaces, and as such, homotopic to a homeomorphism. Thus, after a homotopy, $f$ induces a proper degree $1$ map
$$h = f|: H = \overline{E(k) \setminus T(S)} \to \overline{E(k') \setminus T(S')} = H',$$
where $H$ and $H'$ are handlebodies of genus $2g(S)$, such that $h|: \partial H \stackrel{\cong}{\longrightarrow} \partial H'$ is a homeomorphism. The latter implies that $h_*: \pi_1(H) \to \pi_1(H')$ is surjective, and as $\pi_1(H) \cong \pi_1(H')$ are free, and therefore  Hopfian, $h_*$ is an isomorphism. Now apply Waldhausen's result (Theorem 13.6 of \cite{He}) to conclude that $h$ is homotopic rel $\partial H$ to a homeomorphism. Consequently, the same conclusion holds for $f: E(k)\to E(k')$. Thus $k=k'$.

(3) The inequality $\hat g (k) \ge \hat g (k')$ follows from the equality of immersed and
embedded genus (Corollary 6.18, \cite{Ga}). Suppose then that $k$ is free and
$\hat g (k) = \hat g (k') < \infty$. Fix a proper degree $1$ map $f: E(k) \to E(k')$
 and an incompressible Seifert surface $S'$ for $k'$ with $g(S') = \hat g(k')$.
 The proof of part (1) shows that we can find an incompressible Seifert surface
 $S \subset E(k)$ for $k$ and homotope $f$ to be transverse to $S$ and satisfy $S = f^{-1}(S')$. Then $f$ induces a proper degree $1$ map $S \to S'$ and hence, $\hat g (k) \ge g(S) \geq g(S') = \hat g (k') = \hat g (k)$. Thus (2) implies that $k = k'$.
 \end{proof}
 
\bigskip

An immediate consequence is the following

\begin{Corollary}\label{$1$-domination length in genus}
Suppose $k_0$ is a free knot  and $\hat g(k_0)$ is bounded.
Then for
any $1$-domination sequence $k_0> k_1> .... > k_n$, $n  +\hat g(k_n)\le \hat g(k_0)$. In particular the length of the sequence is at most $\hat g(k_0)$.
\end{Corollary}

\section{Alexander invariant}\label{Alexander}

The proof of the following proposition is styled on classic arguments \cite{Br}. 

\begin{Proposition}\label{splitting}
If $k_1\ge k_2 $, then $\Lambda_{k_1}=\Lambda_{k_2}\oplus \Lambda$ where $\Lambda$ is a $\mathbb Z[t^{\pm 1}]$-module. In particular $\Delta_{k_2}$ divides $\Delta_{k_1}$.
\end{Proposition}

\begin{proof}
 Let $\tilde E(k_i)$ be the infinite cyclic covering of
$E(k_i)$ and $t_i$ be the generator of the deck transformation
group of the infinite cyclic covering. Then $f: E(k_1)\to E(k_2)$ lifts
to a proper degree $1$ map $\tilde f: \tilde E(k_1)\to \tilde
E(k_2)$. We have  
induced homomorphisms $\tilde f_*: H_1(\tilde E(k_1);\mathbb Q)\to
H_1(\tilde E(k_2);\mathbb Q)$ and $\tilde f^*: H^1(\tilde E(k_2);\mathbb Q)\to
H^1(\tilde E(k_1);\mathbb Q)$.

Since knot complements have the homology of the circle, Assertion 5 of [Mi] shows that $H_*(\tilde E(k_1);\mathbb Q)$ is finitely dimensional 
over $\mathbb Q$.  

For each $i$ let $u_i$ be the fundamental class of $H_2(E(k_i), \partial E(k_i); \mathbb Q)$.
There is a duality isomorphism $P_i=u_i\cap : H^1(\tilde E(k_i); \mathbb Q)\to
H_1(\tilde E(k_i); \mathbb Q)$, see \cite[Assertion 9 and Section 4]{Mi}. 

Let $\alpha : H_1(\tilde E(k_2); \mathbb Q)\to H_1(\tilde E(k_1); \mathbb Q)$ be given
by $\alpha(x)=u_1\cap \tilde f^*(P^{-1}_2(x))$ for each $x\in
H_1(\tilde E(k_2); \mathbb Q)$. Then
$$ \tilde f_*\alpha(x) = \tilde f_* (u_1\cap \tilde f^*(P^{-1}_2(x))
=\tilde f^* (u_1)\cap (P^{-1}_2(x))=u_2 \cap (P^{-1}_2(x))=x,$$
Thus $\tilde f_*\alpha$ is the identity on $H_1(\tilde E(k_2);
\mathbb Q)$. It follows that
$$ H_1(\tilde E(k_1);\mathbb Q)\cong H_1(\tilde E(k_2);\mathbb Q) \oplus \text{ker} \tilde f_*.$$
Next we prove that an analogous splitting holds over $\mathbb Z$.

Since $H_1(\tilde E(k_i);\mathbb Z)$ is torsion free, there is an inclusion $\tau_*: H_1(\tilde E(k_i);\mathbb Z) \to H_1(\tilde E(k_i);\mathbb Q)$, 
and since both $\tilde f_*$ and $\alpha$ preserve integer homology, the restriction 
$\tilde f_* \alpha|H_1(\tilde E_2; \mathbb Z)$ is the identity. It follows that
$$ H_1(\tilde E(k_1);\mathbb Z)\cong H_1(\tilde E(k_2);\mathbb Z) \oplus \text{ker} \tilde f_*|.$$
It is also easy to see that $\tilde f_* t_1 = t_2 \tilde f_*$ and
$\alpha t_2 = t_1 \alpha$. Hence the splitting above gives the desired 
splitting of $\mathbb Z[t^{\pm 1}]$ modules.
\end{proof}

An immediate consequence of Propostion \ref{splitting} is that $k_1 \geq k_2$ implies that $\Delta_{k_2}$ divides $\Delta_{k_1}$. This follows from the fact that if
$k$ is a knot, then $H_1(\tilde E(k);\mathbb Q)\cong \Gamma/(p_1(t)) \oplus \ldots \oplus \Gamma/(p_n(t))$ where $p_1(t), \ldots, p_n(t) \in \Gamma = \mathbb Q[t^{\pm 1}]$ 
and $p_1(t) \cdots p_n(t) = \Delta(k)$. Thus if $\Delta_{k_{1}}$ and $\Delta_{k_{2}}$ have the same degree,
then $\Delta_{k_{1}} = \pm \Delta_{k_{2}}$.

One might hope to use band-connected sum and Murasugi sum to
produce examples of 1-dominance: see [Ka2] for definitions.  The
following direct application of Proposition \ref{splitting} shows that this fails in general.

\begin{Example}  
{\rm Figure 2 is  a band connected sum $k$ of the trefoil
knot $3_1$ and the trivial knot with $\Delta_k(t)=1-t^2+t^4$,
which does not have $\Delta_{3_1}(t)=1-t+t^2$ as a factor. It
follows that band connected sum does not $1$-dominate its factors in
general.}
\end{Example}

\begin{center}%
\includegraphics[totalheight=5cm]{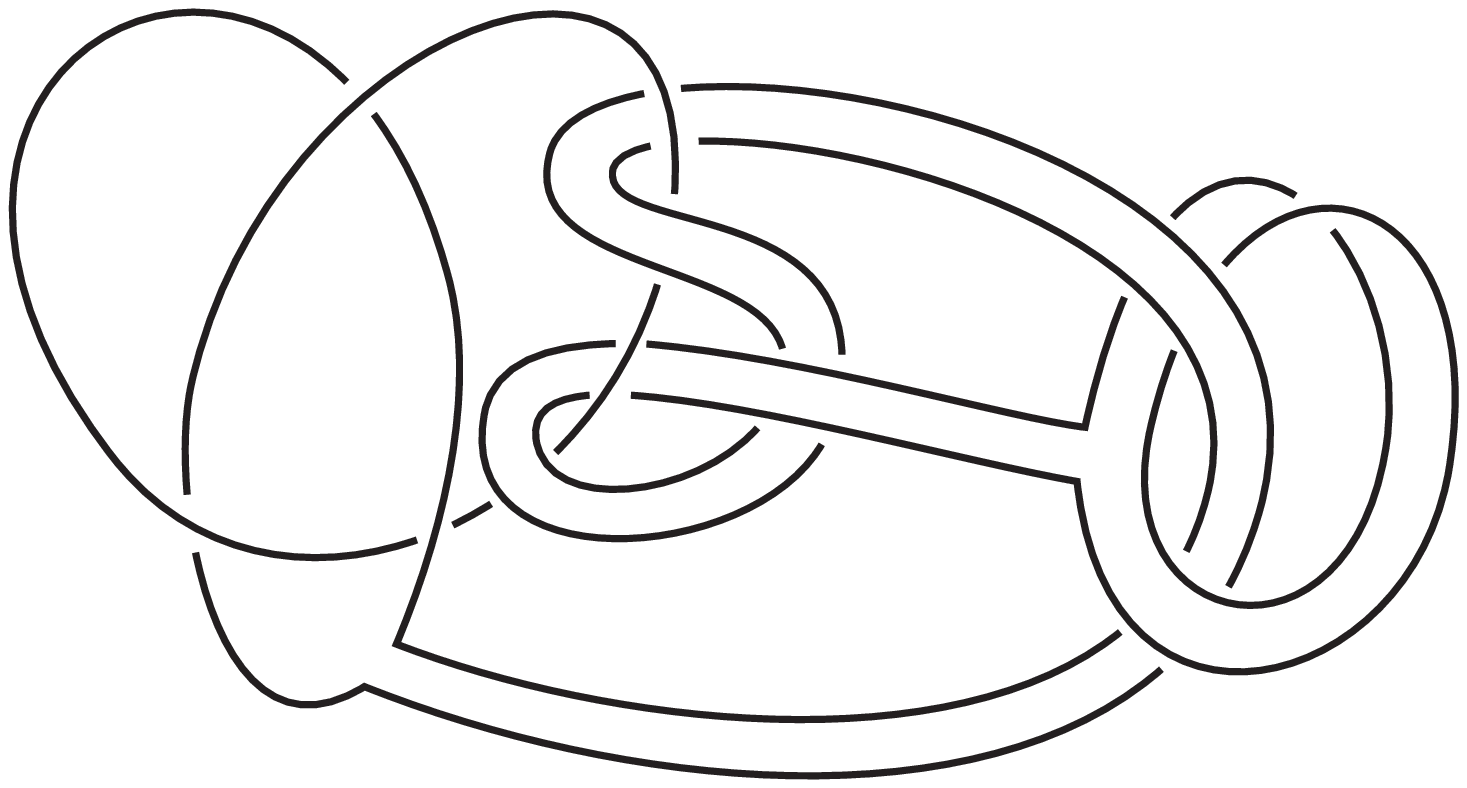}%
\begin{center}%
Figure 2
\end{center}
\end{center}

\begin{Example}  
{\rm Figure 3 is a Murasugi sum $k$ of $5_2$ and $4_1$
with $\Delta_k(t)=2-3t+3t^2-3t^3+2t^4$, which contain neither
$\Delta_{4_1}(t)=1-3t+t^2$ nor $\Delta_{5_2}(t)=2-3t+2t^2$ as a
factor. It follows that Murasugi sum does not,  in general,
$1$-dominate its factors.}
\end{Example}

\begin{center}%
\includegraphics[totalheight=7cm]{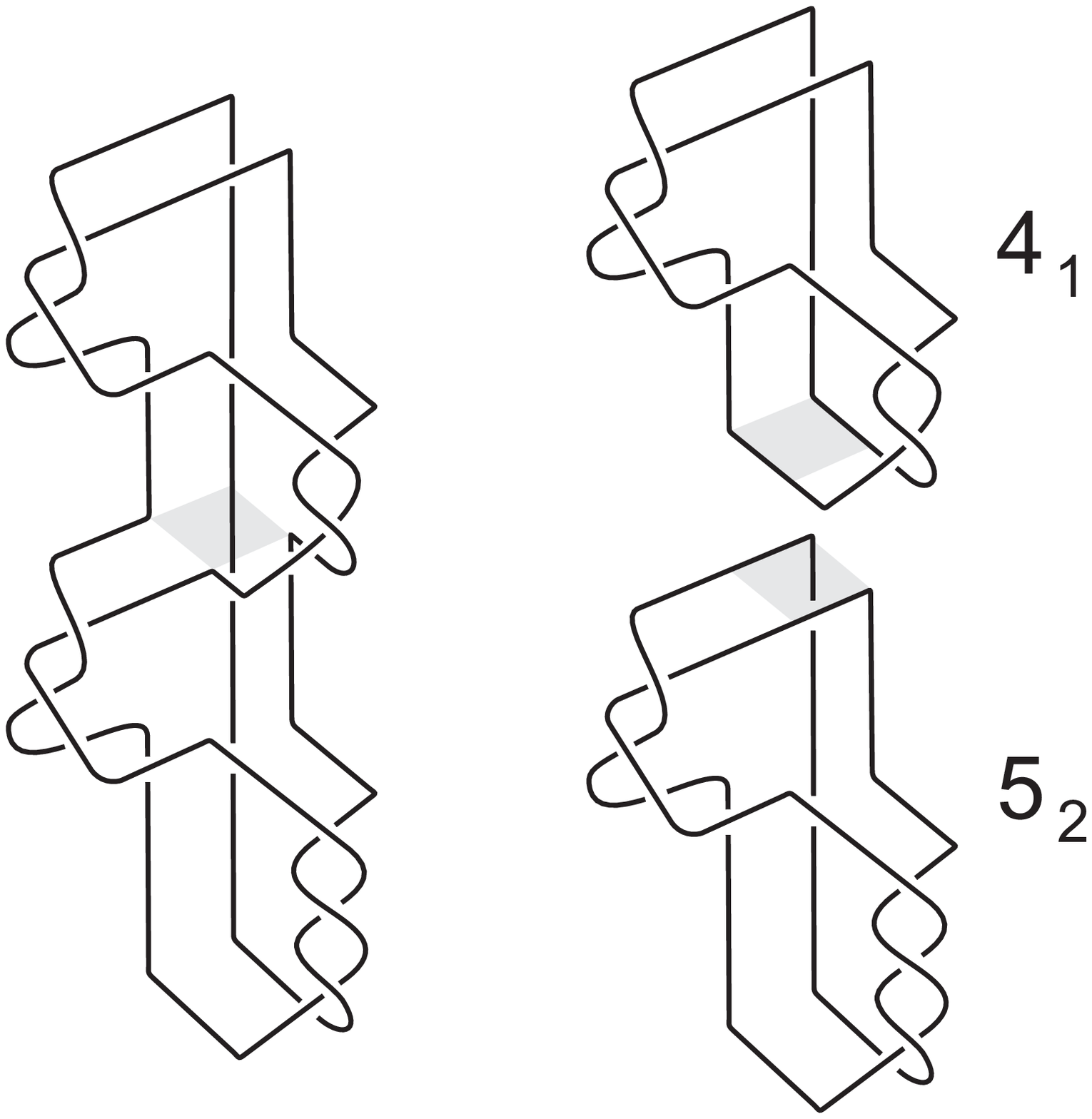}%
\begin{center}%
Figure 3
\end{center}
\end{center}

Referring to the definition of satellite knots in Section \ref{satellite},
if the winding number of $k_p$ in $V$ is $\pm 1$, then there is a
proper degree $1$ map $g: V\setminus N(k_p)\to S^1\times S^1\times
[0,1]$ which is a homeomorphism on the boundaries (see [Du]).
This provides a $1$-domination $k_s \ge k_c$.
The next example shows that $k_s \ge k_c$ need not hold without the assumption of winding number $\pm 1$, as
$\Delta_{k_c}$ does not divide $\Delta_{k_s}$.

\begin{Example}\label{cable}

{\rm Let $k_s$ be the $(2,3)$ cable of the figure-eight knot $k_c =4_1$.  That is, $k_s$
is the satellite of $k_c$ defined by the pattern knot, which is a trefoil, $k_p = 3_1$
with winding number 2 in the solid torus $V$, as in the description above.  The Alexander
polynomials of these knots are
$$ \Delta_{k_p} = 1 - t + t^2, \quad \Delta_{k_c} = 1 -3t +t^2,$$
$$ \Delta_{k_s} = (1 - t - t^2)(1 - t +t^2)(1 + t - t^2).$$
}
\end{Example}

\begin{Remark} {\rm In [p.463, Wan], it is asked whether Jones polynomials will provide
an obstruction to $1$-domination.
In Example \ref{cable}, we have
Jones polynomial:
$$V_{k_s} = t^{-5} - t^{-4} +t + t^3 -t^4 - t^7 +t^8,$$
which is irreducible, and certainly does not have $V_{k_p}$ as a
factor, despite the 1-domination $k_s \geq k_p$.  Therefore the
Jones polynomial does not reflect 1-dominance in a manner
analogous to the Alexander or A-polynomials.  The same may be said of the HOMFLYPT polynomial.}
\end{Remark}

As a final topic in this section, we apply Gordon's approach to ribbon concordance [Go] to prove certain 
$1$-domination rigidity results in terms of  Alexander polynomials.

Let $G$ be a group and let $H\subset G$ be a subgroup. Assume that
$p$ is a fixed integer either equal to $0$ or a prime number.
Define $G \natural H$ to be the subgroup of $G$ generated by all
the elements of the form $[x,y]z^p$ for $x \in G, \, y \in H$ and
$z \in H$. The {\it lower p-central series} of $G$ is defined as
follows: $G_0 = G$, $G_{\alpha +1} = G \natural G_{\alpha}$ and
$G_{\beta} = \cap_{\alpha <\beta}G_{\alpha}$ if $\beta$ is a limit
ordinal. We say that $G$ is {\it transfinitely p-nilpotent} if $G_{\alpha}
= \{1\}$ for some ordinal $\alpha$. In particular the group is residually $p$-nilpotent (or residually $p$ for short)
if and only if $G_{\omega} = \{1\}$.

\begin{Definition} {\rm \cite{Go} 
{\rm A knot $k \subset S^3$ is {\it transfinitely $p$-nilpotent} 
if its commutator subgroup $[\pi_1E(k),\pi_1E(k)]$ is transfinitely
$p$-nilpotent.}}
\end{Definition}

\medskip

The class of transfinitely $p$-nilpotent knots contains 
$2$-bridge knots, fibred knots and, when $p > 0$, alternating knots $k$ for which the leading coefficient of $\Delta_{k}$ is
a power of $p$. Moreover  it has been observed by Gordon that the property of being transfinitely p-nilpotent is preserved by connected sum and cabling, see \cite{Go}.

For a polynomial $P$, we use $d^o(P)$ to denote the degree of $P$.
The following proposition is essentially \cite[Lemma 3.4]{Go}

\begin{Proposition}\label{nilpotent} Let $k_1$ and
$k_2$ be two knots in $S^3$ such that $k_1 \geq k_2$. If $k_1$ is
transfinitely $p$-nilpotent for some $p$ and $d^{o}(\Delta_{k_{1}}) = d^{o}(\Delta_{k_{2}})$, then
$k_1 = k_2$.
\end{Proposition}

\begin{proof}
The proper degree $1$ map $f: E(k_1) \to E(k_2)$ induces an epimorphism $f_*: \pi_1E(k_1)
\to \pi_1E(k_2)$. It induces an epimorphism $\hat f_* :[\pi_1E(k_1),\pi_1E(k_1)] \to
[\pi_1E(k_2),\pi_1E(k_2)]$. 

For a knot $k \subset S^3$ it is well-known that $H_1([\pi_1E(k),\pi_1E(k)];\Z)$ is torsion-free and $H_2([\pi_1E(k),\pi_1E(k)];\Z) = 0$. Thus 
$H_2([\pi_1E(k),\pi_1E(k)];\F_p) = \{0\}$ where $\F_p = \mathbb Q$ when $p = 0$ and $\Z/p\Z$ otherwise. It is also known that for a field $\mathbb F$, 
$\text{rank}\,(H_1([\pi(k),\pi(k)];\F)) = d^{o}(\Delta_{k})$. Therefore our hypotheses imply that 
the epimorphism $\hat f_*$ induces an isomorphism 
$$\hat f_{\sharp}: H_1([\pi_1E(k_1),\pi_1E(k_1)];\F_p) \to
H_1([\pi_1E(k_2),\pi_1E(k_2)];\F_p).$$
Stallings' theorem \cite[Theorem 3.4]{Sta} implies that for every ordinal $\alpha$, $\hat
f_*$ induces an isomorphism

{\small $$[\pi_1E(k_1),\pi_1E(k_1)]/[\pi_1E(k_1),\pi_1E(k_1)]_{\alpha}
\to [\pi_1E(k_2),\pi_1E(k_2)]/[\pi_1E(k_2),\pi_1E(k_2)]_{\alpha}.$$}

\noindent By hypothesis we have $[\pi_1E(k_1),\pi_1E(k_1)]_{\alpha} =
\{1\}$ for some ordinal $\alpha$, so the epimorphism $\hat f_* :[\pi_1E(k_1),\pi_1E(k_1)] \to
[\pi_1E(k_2),\pi_1E(k_2)]$ is in fact an isomorphism. Since $f$ induces an isomorphism $f_{\sharp} : H_1(\pi_1E(k_1);\Z) \to
H_1(\pi_1E(k_2);\Z)$, it follows that the epimorphism $f_*: \pi_1E(k_1)
\to \pi_1E(k_2)$ is an isomorphism. Finally, since this isomorphism preserves the
peripheral structures of the two knots, the two knots are the same by Waldhausen [H, Chap 13]
and [GL].
\end{proof}

Since alternating knots $k$ for which the leading coefficient of $\Delta_{k}$ is
a power of $p$ are transfinitely $p$-nilpotent, the followings are straightforward consequences of Propositions \ref{splitting} and \ref{nilpotent}:

\begin{Corollary}\label{1-domination alternating}
Suppose that $k_1\ge k_2$ where $k_1$ is an alternating knot such that the leading
coefficient of $\Delta_{k_{1}}$ is a power of a prime number and $d^{o}(\Delta_{k_{1}}) =
d^{o}(\Delta_{k_{2}})$. Then $k_1 = k_2$.
\end{Corollary}

\begin{Corollary}\label{sequence alternating}
Suppose $k_0$ is an alternating knot such that the leading
coefficient of $\Delta_{k_{0}}$ is a power of a prime number
Then 
any $1$-domination sequence $k_0> k_1> .... > k_n >....$, contains at most $d^{o}(\Delta_{k_{0}})$ alternating knots.
\end{Corollary}


\begin{Question}\label{Q-nilponent}
Is each alternating knot  transfinitely $p$-nilpotent?
\end{Question}

\begin{Question}\label{alternation}
Suppose that $k$ is alternating. 
Does $k\ge k'$ imply that $k'$ is alternating?

\end{Question}

\begin{Remark} $\;$ \\
{\rm (1) A positive answer to Question \ref{Q-nilponent} implies that if $k_1$ is an alternating knot, $k_1\ge k_2$, and  $d^{o}(\Delta_{k_{1}}) =
d^{o}(\Delta_{k_{2}})$, then $k_1 = k_2$.

\noindent (2) A positive answer to both Question \ref{Q-nilponent} and Question \ref{alternation}
implies that any $1$-domination sequence of knots starting with an alternating knot $k$
has length at most $d^{o}(\Delta_{k})$.}
\end{Remark}

\medskip\noindent
MICHEL BOILEAU, michel.boileau@univ-amu.fr\\
{\it Aix-Marseille Universit\'e, CNRS, Centrale Marseille, I2M, UMR 7373,
13453 Marseille, France}

\medskip\noindent
STEVEN BOYER, boyer.steven@uqam.ca\\
{\it Departement de mathematiques,  UQAM, Montreal H3C 3P8 Canada}

\medskip\noindent
DALE ROLFSEN, rolfsen@math.ubc.ca\\
{\it Department of Mathematics, The University of British
Columbia, V6T 1Z2 Canada }

\medskip\noindent
SHICHENG WANG, wangsc@math.pku.edu.cn\\
{\it LMAM, Department of Mathematics, Peking University, Beijing
100871 China }

\end{document}